\documentclass{amsart}

\usepackage{enumerate,color,amsmath,amssymb,setspace,hyperref}
\usepackage[shortalphabetic]{amsrefs}

\renewcommand{\baselinestretch}{1.2}

\newcommand\A{\ensuremath{\mathcal{A}}}
\newcommand\Q{\ensuremath{\mathbb{Q}}}
\newcommand\Z{\ensuremath{\mathbb{Z}}}

\newcommand\C{\ensuremath{\mathcal{C}}}
\newcommand\G{\Gamma}
\newcommand\La{\Lambda}
\newcommand\al{\alpha}
\newcommand\be{\beta}
\newcommand\ga{\gamma}
\newcommand\ep{\epsilon}
\newcommand{\la}{\lambda}
\newcommand{\So}{\Sigma_0}

\DeclareMathOperator\SL{SL}

\DeclareMathOperator\Stab{Stab}

\newcommand\I{^{-1}}

\newcommand{\ger}[1]{\mathfrak{#1}}
\newcommand{\gen}{\langle}
\newcommand{\by}{\rangle}

\theoremstyle{plain}
\newtheorem{theorem}{Theorem}[section]
\newtheorem*{theorem*}{Theorem}
\newtheorem{lemma}[theorem]{Lemma}
\newtheorem{proposition}[theorem]{Proposition}
\newtheorem{corollary}[theorem]{Corollary}
\newtheorem*{corollary*}{Corollary}

\theoremstyle{remark}
\newtheorem{remark}[theorem]{Remark}

\theoremstyle{definition}
\newtheorem{definition}[theorem]{Definition}

\numberwithin{equation}{section}

\begin{document}

\title[Transitivity of group actions on buildings from Chevalley groups]{Strongly and Weyl transitive group actions on buildings arising from Chevalley groups}

\author{Peter Abramenko}
\address{Department of Mathematics\\
University of Virginia\\
Charlottesville, VA 22904}
\email{pa8e@virginia.edu}

\author{Matthew C. B. Zaremsky}
\address{Department of Mathematics\\
University of Virginia\\
Charlottesville, VA 22904}
\email{mcz5r@virginia.edu}

\begin{abstract}
\singlespacing
Let $K$ be a field and $\ger{g}(K)$ a Chevalley group (scheme) over $K$. Let $(B,N)$ be the standard spherical $BN$-pair in $\ger{g}(K)$, with $T=B\cap N$ and Weyl group $W=N/T$. We prove that there exist non-trivial elements $w\in W$ such that all representatives of $w$ in $N$ have finite order. This allows us to exhibit examples of subgroups of $\ger{g}(\Q_p)$ that act Weyl transitively but not strongly transitively on the affine building $\Delta$ associated with $\ger{g}(\Q_p)$. Such examples were previously known only in the case when $\ger{g}(\Q_p)=\SL_2(\Q_p)$ and $\Delta$ is a tree (see \cite{abramenko07:_trans_properties}).
\end{abstract}

\maketitle

\section{Introduction}
\label{sec:introduction}
In building theory, there are two important concepts of transitivity that are stronger than chamber transitive actions, namely strongly transitive and Weyl transitive actions (for the precise definitions, we refer to Section~\ref{sec:facts} below). Strongly transitive group actions on (thick) buildings are equivalent to BN-pairs, which is reason enough for studying them. However, if one considers buildings from the W-metric point of view, there is another natural notion of transitivity, which was called ``Weyl transitivity'' in \cite{abramenko07:_trans_properties}. Strong transitivity is always defined with respect to a chosen (not necessarily complete) system of apartments of the building in question, whereas the definition of Weyl transitivity does not refer to apartments. It is well known that strong transitivity (with respect to any apartment system) always implies Weyl transitivity, and that these two notions are equivalent for buildings of spherical type. It was also expected that for non-spherical buildings Weyl transitivity is strictly weaker than strong transitivity. However, the first explicit examples of this type were documented only a few years ago in \cite{abramenko07:_trans_properties}, following some suggestions made by J. Tits in \cite{tits92:_twin_kac_moody}*{Section~3.1,~Example~(b)}. Tits suggested to analyze the actions of anisotropic groups over global fields on suitable Bruhat-Tits buildings, and in \cite{abramenko07:_trans_properties} this was done in the simplest case, namely for the norm 1 group of a quaternion division algebra $D$ over $\Q$ acting on the Bruhat-Tits tree $T_p$ of $\SL_2(\Q_p)$ for suitable $p$. Provided that -1 is not a square in $D$, this yielded the first explicit examples of Weyl transitive actions which are not strongly transitive with respect to any apartment system. However, trees are rather special buildings, and the question remained whether one can produce examples of groups acting Weyl transitively but not strongly transitively on buildings of arbitrary (affine) type.

It is the goal of this paper to present some examples of this kind. Our approach here will not use anisotropic algebraic groups but will rather generalize a second type of counter-example discussed in Section 6.10.2 of \cite{abramenko08:_build_theory_apps}. Here, dense subgroups of $\SL_2(\Q_p)$ are exhibited that do not act strongly transitively on $T_p$. By a general principle, formulated and
proved as Proposition 3.4 in \cite{abramenko07:_trans_properties} and restated as Lemma~\ref{dense-subgroups} below, the density immediately implies that these subgroups act Weyl transitively. The argument given in \cite{abramenko08:_build_theory_apps}*{Section~6.10.2} that shows that the actions are not strongly transitive appears to be rather special and is restricted to the tree case. Just based on this argument, it is not clear how to show for higher dimensions or different types that a given Weyl transitive action that is not strongly transitive with respect to a {\em certain} apartment system is in fact not strongly transitive with respect to {\em any} apartment system. In this paper we will generalize the group $\SL_2(\Q_p)$ to an arbitrary $p$-adic Chevalley group $\ger{g}(\Q_p)$. The generalization of the argument given in \cite{abramenko08:_build_theory_apps}*{Section~6.10.2} that does the job is the following fact, which turns out to be true for all $p$-adic Chevalley groups $\ger{g}(\Q_p)$:

\begin{proposition}\label{prop1}

If a subgroup of $\ger{g}(\Q_p)$ acts strongly transitively on the corresponding Bruhat-Tits building $\Delta$, then it contains nontrivial elements of finite order.

\end{proposition}

An explanation of how $\Delta$ arises from $\ger{g}(\Q_p)$ will be given in Section~\ref{sec:chev_gps}, and a proof of this proposition will be given in Section~\ref{sec:strong-trans-torsion}. Then, using arguments in Section~\ref{sec:weyl-trans-torfree}, one obtains the following:

\begin{proposition}\label{prop2}

$\ger{g}(\Q_p)$ has (many) dense torsionfree subgroups. The action of any such subgroup on $\Delta$ is Weyl transitive but not strongly transitive with respect to any apartment system of $\Delta$.

\end{proposition}


\section{Some facts about strongly and Weyl transitive actions}
\label{sec:facts}

Let $\Delta$ be a thick building with associated Coxeter system $(W,S)$ of finite rank $|S|$. Denote by $\C$ the set of chambers of $\Delta$ and by $\delta :  \C \times \C \longrightarrow W$ the associated Weyl distance function. Recall that a group $H$ is said to act \emph{Weyl transitively} on $\Delta$ if it acts on $\C$ preserving $\delta$ and such that for any given $w \in W$, the action of $H$ on $\{ (C,C') \in \C \times \C \mid \delta(C,C') = w \}$ is transitive. In particular, Weyl transitive actions are type-preserving and chamber transitive. It can be shown that a group $H$ acts Weyl transitively on some thick building if and only if $H$ admits a ``Tits subgroup;'' see \cite{abramenko08:_build_theory_apps}*{Proposition~6.34~and~Definitions~6.39~and~6.45}.

\

Let $\overline{\A}$ be the complete system of apartments of $\Delta$. By definition, a subset $\A \subseteq \overline{\A}$ is a system of apartments of $\Delta$ if for any two $C,C' \in \Delta$, there is a $\Sigma \in \A$ which contains $C$ and $C'$. A type-preserving action of a group $H$ on $\Delta$ is called \emph{strongly transitive with respect to $\A$} if it is transitive on the set $\{ (C, \Sigma) \in \C \times \A \mid C \in \Sigma \}$. The action of $H$ on $\Delta$ is called \emph{strongly transitive} if it is strongly transitive with respect to \emph{some} system of apartments $\A$ of $\Delta$. It is clear that $H$ acts strongly transitively on $\Delta$ with respect to a system of apartments $\A$ if it acts chamber transitively, and for some (and hence any) chamber $C$ the stabilizer $\Stab_{H}(C)$ acts transitively on $\{ \Sigma \in \A \mid C \in \Sigma \}$. Equivalently, $H$ acts transitively on $\A$, and for some (and hence any) $\Sigma \in \A$ the stabilizer $\Stab_{H}(\Sigma)$ acts transitively on the set of chambers of $\Sigma$. It is well known that a group $H$ acts strongly transitively on a thick building if and only if $H$ admits a BN-pair. If $\Delta$ is spherical, it is also well known that $\overline{\A}$ is the only system of apartments of $\Delta$, and that the following three statements for the action of a group $H$ on $\Delta$ are equivalent: (i) $H$ acts Weyl transitively; (ii) $H$ acts transitively on pairs of opposite chambers of $\Delta$; (iii) $H$ acts strongly transitively. In the following lemma we collect some further (easy) results which relate strong and Weyl transitivity. The proofs can be found in \cite{abramenko07:_trans_properties}*{Section~3} or in \cite{abramenko08:_build_theory_apps}*{Section~6.1.3}.

\

\begin{lemma}\label{Weyl}
Let $H$ be a group acting on the building $\Delta$.
\begin{enumerate}[\rm (1)]
\item
If $H$ acts strongly transitively, then it also acts Weyl transitively.

\item
If $H$ acts Weyl transitively, then $H\Sigma = \{ h\Sigma \mid h \in H \}$ is a system of apartments of $\Delta$ for any $\Sigma \in \overline{\A}$.

\item
$H$ acts strongly transitively if and only it acts Weyl transitively and there exists an apartment $\Sigma \in \overline{\A}$ such that $\Stab_{H}(\Sigma)$
acts chamber transitively on $\Sigma$.
\end{enumerate}

\end{lemma}

In view of (3) we make the following definition:

\begin{definition}
A type-preserving action of a group $H$ on a building $\Delta$ is called \emph{weakly transitive} if there exists an apartment $\Sigma \in \overline{\A}$ such that $\Stab_{H}(\Sigma)$ acts chamber transitively on $\Sigma$.
\end{definition}

It is clear that an action is strongly transitive if and only if it is both Weyl and weakly transitive. In order to find Weyl transitive actions that are not strongly transitive, one needs a group-theoretic criterion equivalent to weak transitivity. We will formulate such a criterion in the following situation, which we shall assume for the rest of this section:

\

Let $G$ be a group acting strongly transitively on $\Delta$ with respect to $\overline{\A}$. Fix an apartment $\Sigma_0$ in $\overline{\A}$, and set $N = \Stab_{G}(\Sigma_0)$,
$T = \{ t \in N \mid tC = C$ for all chambers $C$ in $\Sigma_0 \}$. Note that $N/T$ can be identified with the group of type-preserving automorphisms of $\Sigma_0$, and hence with $W$.

\

\begin{lemma}\label{criterion}
For a subgroup $H$ of $G$, the following are equivalent:
\begin{enumerate}[\rm (i)]
\item
H acts weakly transitively on $\Delta$.
\item
There exists an element $g \in G$ such that $g(nT)g^{-1} \cap H \neq \emptyset$ for all $n \in N$.
\end{enumerate}
\end{lemma}

\begin{proof}
Let $\Sigma$ be any element of $\overline{\A}$. By assumption, there exists $g \in G$ such that $\Sigma  = g\Sigma_0$. Hence $\Stab_{G}(\Sigma) = gNg^{-1}$, and the pointwise fixer of $\Sigma$ in $G$ is $gTg^{-1}$. So $\Stab_{H}(\Sigma) = gNg^{-1} \cap H$, and this group acts transitively on the set $\C(\Sigma)$ of chambers of $\Sigma$ if and only if $\Stab_{H}(\Sigma)(gTg^{-1}) = gNg^{-1}$. (Here we use that $gNg^{-1}/gTg^{-1}$, which can be identified with the group of type-preserving automorphisms of $\Sigma$, acts simply transitively on $\C(\Sigma)$.) But $(gNg^{-1} \cap H)(gTg^{-1}) = gNg^{-1}$ if and only if each coset $(gng^{-1})(gTg^{-1})$ (with $n \in N$) in $gNg^{-1}/gTg^{-1}$ has a representative in $H$, i.e. if and only if $gnTg^{-1} \cap H \neq \emptyset$ for all $n \in N$.
\end{proof}

\begin{corollary}\label{torcoset}
If there exists an element $n_0 \in N \setminus T$ such that all elements of $n_0T$ have finite order in $G$, then no torsionfree subgroup of $G$ acts weakly transitively on $\Delta$. \qed
\end{corollary}

In the following sections we shall apply this set-up to the Chevalley group $G = \ger{g}(\Q_p)$ and the associated affine building $\Delta$, in which case an $n_0$ as in Corollary~\ref{torcoset} can be found. On the other hand, we shall exhibit torsionfree subgroups of $\ger{g}(\Q_p)$ which still act Weyl transitively on $\Delta$. In order to establish the latter, we apply Proposition 3.4 of \cite{abramenko07:_trans_properties}. For the convenience of the reader, we restate this result below.

\begin{lemma}\label{dense-subgroups}
Suppose that $G$ is a topological group and that the stabilizer $B$ of some chamber $C$ of $\Delta$ is an open subgroup of $G$. Then any dense subgroup $H$ of $G$ acts Weyl transitively on $\Delta$.
\end{lemma}

\section{Chevalley Groups and VRGD systems}
\label{sec:chev_gps}

Since the current group of interest is an arbitrary Chevalley group, in this section some background information about Chevalley groups is collected. Chevalley groups will also be established as examples of \emph{RGD systems}, and $p$-adic Chevalley groups as examples of \emph{VRGD systems}. For a more comprehensive review of Chevalley groups, one may consult \cite{steinberg67:_chev_gps}, and for a quick overview (the notation of which we will use here) see \cite{abramenko08:_build_theory_apps}*{Section~7.9.2}. For an overview of RGD (root group data) systems one may consult \cite{abramenko08:_build_theory_apps}*{Chapter~7}, and a reference for what we will call VRGD (valuated root group data) systems can be found in \cite{weiss09:_struct_affine_build}*{Chapter~3}.

Given a complex semisimple Lie algebra $\ger{g}$ with root system $\Phi$, there is a family of groups $\ger{g}(K)=\ger{g}(\Phi,\La,K)$ parameterized by $K$ and $\La$. Here, $K$ is an arbitrary field and $\La\supseteq\Phi$ is a weight lattice arising from some faithful finite-dimensional representation $V$ of $\ger{g}$. Let $V=\bigoplus_{\ga}V_{\ga}$ be the weight space decomposition of $V$, and let $E$ be the Euclidean space spanned by the roots, with Euclidean inner product $\{,\}$. Also let $s_{\al}$ denote the reflection in $E$ about the hyperplane orthogonal to $\al$, and define $\gen v,v'\by:=2\{v,v'\}/\{v',v'\}$. Note that $\gen v,v\by=2$ for all $v\in E$. The group $\ger{g}(K)$ has generators $x_{\al}(\la)$ for $\al\in\Phi$, $\la\in K$, and a series of relations that we will reference as they become necessary. For notational convenience, we define some important elements of $\ger{g}(K)$. For $\al\in\Phi$, $\la\in K^*$, let
$$m_{\al}(\la):=x_{\al}(\la)x_{-\al}(-\la\I)x_{\al}(\la)\textnormal{ and }h_{\al}(\la):=m_{\al}(\la)m_{\al}(1)\I.$$

Define for each root $\al$ a subgroup $U_{\al}:=\{x_{\al}(\la)|\la\in K\}$. As proved in \cite{steinberg67:_chev_gps}*{Corollary~1~to~Lemma~18}, the map $(K,+)\rightarrow U_{\al}$ given by $\la\mapsto x_{\al}(\la)$ is an isomorphism. Also, for fixed $\al$ one has by \cite{steinberg67:_chev_gps}*{Lemma~28(a)} that $h_{\al}(\la)h_{\al}(\mu)=h_{\al}(\la\mu)$. In this way, the algebraic structure of $K$ is reflected in the structure of the Chevalley group. Now we will show that if $K$ has a discrete valuation, the additional structure given by the valuation can also be encoded into the Chevalley group.

First we will see that $(\ger{g}(K),(U_{\al})_{\al\in\Phi})$ is an RGD system. A pair $(G,(U_{\al})_{\al\in\Phi})$, where $(U_{\al})_{\al\in\Phi}$ is a family of subgroups of $G$, is called an RGD system provided the following axioms hold.

\noindent \textbf{(RGD0):} For each $\al\in\Phi$, $U_{\al}\neq\{1\}$.

\noindent \textbf{(RGD1):} For all $\al,\be\in\Phi$ with $\al\neq\pm\be$, $\displaystyle[U_{\al},U_{\be}]\subseteq\prod_{\ga\in(\al,\be)}U_{\ga}$, where $(\al,\be)$ is defined as in \cite{abramenko08:_build_theory_apps}*{Section~7.7.2}

\noindent \textbf{(RGD2):} For each $\al\in\Phi$ there is a function $m:U_{\al}^*\rightarrow G$ such that for $\al\in\Phi$ and $u\in U_{\al}^*$, $m(u)\in U_{-\al}uU_{-\al}$ and $m(u)U_{\be}m(u)\I=U_{s_{\al}(\be)}$.

\noindent \textbf{(RGD3):} For each fundamental root $\al$, $U_{-\al}\not\leq U_+$, where $U_+=\gen U_{\al}|\al\in\Phi^+\by$.

\noindent \textbf{(RGD4):} $G=T\gen U_{\al}|\al\in\Phi\by$, where $\displaystyle T=\bigcap_{\al\in\Phi}N_G(U_{\al})$.

\noindent A proof that $(\ger{g}(K),(U_{\al})_{\al\in\Phi})$ satisfies the axioms is given in \cite{abramenko08:_build_theory_apps}*{Section~7.9.2}, though here the axioms are in a slightly different (equivalent) form.

We now assume that $K$ has a (surjective) discrete valuation $\nu:K\twoheadrightarrow\Z\cup\{\infty\}$. In later sections we will use $K=\Q_p$ and $\nu=\nu_p$. Let (V1) denote the property $\nu(\la\mu)=\nu(\la)+\nu(\mu)$, let (V2) denote the property $\nu(\la+\mu)\geq\max(\nu(\la),\nu(\mu))$, let (V3) denote the property $\nu(c)\geq0$ for $c\in\Z$, and let (V4) denote the property $\nu(-\la)=\nu(\la)$.

The family of maps $(\phi_{\al})_{\al\in\Phi}$ with $\phi_{\al}:U_{\al}^*\rightarrow\Z$ is called a \emph{root group valuation} provided that the following axioms hold:

\textbf{(VRGD0):} Each $\phi_{\al}$ is surjective.

\textbf{(VRGD1):} For each $\al\in\Phi$ and each $k\in\Z$, $U_{\al,k}:=\gen u\in U_{\al}|\phi_{\al}(u)\geq k\by$ is a subgroup of $U_{\al}$, where $\phi_{\al}(1)$ is considered to be $\infty$.

\textbf{(VRGD2):} For all $\al,\be\in\Phi$ with $\al\neq\pm\be$, $\displaystyle[U_{\al,k},U_{\be,\ell}]\subseteq\prod_{\ga\in(\al,\be)}U_{\ga,p_{\ga}k+q_{\ga}\ell}$, where $p_{\ga}$ and $q_{\ga}$ are as defined in \cite{weiss09:_struct_affine_build}*{Chapter~3}.

\textbf{(VRGD3):} For $\al,\be\in\Phi$, $u\in U_{\al}^*$, $x\in U_{\be}^*$, one has that $\phi_{s_{\al}(\be)}(m(u)\I xm(u))-\phi_{\be}(x)$ is independent of $x$, where $m:U_{\al}^*\rightarrow \ger{g}(K)$ is as defined in (RGD2).

\textbf{(VRGD4):} For $\al\in\Phi$,  $u\in U_{\al}^*$, $x\in U_{\al}^*$, we have that $\phi_{-\al}(m(u)\I xm(u))-\phi_{\al}(x)=-2\phi_{\al}(u)$, independent of $x$.

\begin{definition}\label{VRGD_def}
Let $G$ be a group with a family of subgroups $(U_{\al})_{\al\in\Phi}$. Let $(\phi_{\al})_{\al\in\Phi}$ be a family of maps $\phi_{\al}:U_{\al}\rightarrow\Z$. If $(G,(U_{\al})_{\al\in\Phi})$ is an RGD system and $(\phi_{\al})_{\al\in\Phi}$ is a root group valuation then $(G,(U_{\al})_{\al},(\phi_{\al})_{\al})$ is called a \emph{VRGD system}.
\end{definition}

Now, in the particular case of $G=\ger{g}(K)$ where $K$ has (surjective) discrete valuation $\nu$, define for each $\al\in\Phi$ a map $\phi_{\al}:U_{\al}\rightarrow\Z$ by $x_{\al}(\la)\mapsto \nu(\la)$. Since $U_{\al}\cong(K,+)$ via $x_{\al}(\la)\mapsto\la$, $\phi_{\al}$ is clearly well-defined. It is easy to check that $(G,(U_{\al})_{\al\in\Phi},(\phi_{\al})_{\al\in\Phi})$ is a VRGD system.

\begin{proposition}\label{chev_is_VRGD}
$(G,(U_{\al})_{\al\in\Phi},(\phi_{\al})_{\al\in\Phi})$ is a VRGD system.
\end{proposition}

\begin{proof}
(VRGD0) holds trivially. (VRGD1) follows from (V2), (V4), and the canonical isomorphism $(K,+)\rightarrow U_{\al}$. (VRGD2) follows from the Chevalley relation (R2) found after Lemma 20 of \cite{steinberg67:_chev_gps}, and from (V1) and (V3). 

Next we check (VRGD3). By \cite{abramenko08:_build_theory_apps}*{Equation~7.36}, we have that
$$m_{\al}(\la) x_{\be}(\mu)m_{\al}(\la)\I=x_{s_{\al}(\be)}(\pm\la^{\gen s_{\al}(\be),\al\by}\mu)$$
for any $\al,\be\in\Phi$, $\la\in K^*$, $\mu\in K$. Denote this relation by (R). Let $u=x_{\al}(\la)\in U_{\al}$ and set $m(u)=m_{-\al}(-\la\I)$, as in the proof that $(G,(U_{\al})_{\al\in\Phi})$ is an RGD system. Also let $x=x_{\be}(\mu)$. We claim that the integer given by $\phi_{s_{\al}(\be)}(m(u)\I xm(u))-\phi_{\be}(x)$ is independent of $\mu$. We know that this quantity equals
\begin{align*}
\phi_{s_{\al}(\be)}(m_{-\al}(-\la\I)\I x_{\be}(\mu)m_{-\al}(-\la\I))-\nu(\mu)
\\
=\phi_{s_{\al}(\be)}(x_{s_{\al}(\be)}(\pm\la^{-\gen s_{\al}(\be),-\al\by}\mu))-\nu(\mu)
\end{align*}
by (R), since $s_{-\al}=s_{\al}$. This equals
$$\nu(\pm\la^{\gen s_{\al}(\be),\al\by}\mu)-\nu(\mu)=\nu(\pm\la^{\gen s_{\al}(\be),\al\by})$$
by (V1) and (V4). This quantity is indeed independent of $\mu$ and so (VRGD3) follows.

Note that when $\al=\be$ we have $s_{\al}(\be)=s_{\al}(\al)=-\al$. Thus
\begin{align*}
\nu(\pm\la^{\gen s_{\al}(\be),\al\by})=\nu(\pm\la^{-\gen \al,\al\by})=\nu(\pm\la^{-2})=-2\nu(\la)=-2\phi_{\al}(u)
\end{align*}
by (V1) and (V4), so (VRGD4) holds. Thus, all the axioms are satisfied and $(G,(U_{\al})_{\al\in\Phi},(\phi_{\al})_{\al\in\Phi})$ is a VRGD system.
\end{proof}

Since $(G,(U_{\al})_{\al\in\Phi},(\phi_{\al})_{\al\in\Phi})$ is a VRGD system one can use \cite{weiss09:_struct_affine_build}*{Theorem~14.38} to conclude that $G$ has an affine $BN$-pair $(B_a,N)$, with affine Weyl group $W_a=N/B_a\cap N$. By \cite{abramenko08:_build_theory_apps}*{Theorem~6.56} one gets an affine building $\Delta$ on which $G$ acts strongly transitively with respect to $G.\So$. If one makes the further assumption that $K$ is complete with respect to the metric induced by $\nu$, then by \cite{weiss09:_struct_affine_build}*{Theorem~17.7~and~17.9} one knows that $G$ in fact acts strongly transitively on $\Delta$ with respect to the \emph{complete} apartment system. Since $\Q_p$ is complete, one can apply Corollary~\ref{torcoset} to the setup of $\ger{g}(\Q_p)$ acting on the corresponding building $\Delta$.

Of course, by virtue of $(G,(U_{\al})_{\al\in\Phi})$ being an RGD system, there is also a spherical $BN$-pair $(B,N)$ with spherical Weyl group $W=N/B\cap N$. Looking at the construction of the affine and spherical $BN$-pairs, one sees that the subgroup $N$ is the same in both cases, namely $N=\gen m_{\al}(\la)|\al\in\Phi,\la\in K^*\by$; see \cite{steinberg67:_chev_gps}*{Lemma~22} and \cite{weiss09:_struct_affine_build}*{14.3}. Also, $T_a\subseteq T$, where $T_a=B_a\cap N=\gen h_{\al}(\la)|\al\in\Phi,\la\in A^{\times}\by$ and $T=B\cap N=\gen h_{\al}(\la)|\al\in\Phi,\la\in K^*\by$, with $A$ the valuation ring of $K$.

By \cite{steinberg67:_chev_gps}*{Lemma~22(b,c)} $W$ is isomorphic to the Weyl group of the root system $\Phi$ (thus the name), and so any $w\in W$ acts as an orthogonal transformation on the Euclidean vector space $E$ spanned by $\Phi$. We will see in the next section that it will be useful to think of $W$ interchangeably as both the quotient $N/T$ and as the Weyl group of the root system $\Phi$, acting on $E$. We also note that $W_a$ is in fact the affine Weyl group associated to $\Phi$, though we will not use this explicitly.

\section{Strongly transitive subgroups of $\ger{g}(\Q_p)$ have torsion}
\label{sec:strong-trans-torsion}

The goal of this section is to establish Proposition~\ref{prop1}. Let $G=\ger{g}(\Q_p)$, with affine $BN$-pair $(B_a,N)$ and affine Weyl group $W_a$, and let $\Delta$ be the canonical affine building associated to $G$ as described in the previous section. We restate the proposition:

\begin{proposition}\label{mainprop}
No torsionfree subgroup of $G$ acts strongly transitively on $\Delta$.
\end{proposition}

By Corollary~\ref{torcoset} it suffices to exhibit a non-trivial element $w$ in $W_a=N/T_a$ such that all representatives of $w$ in $N$ have finite order in $G$. Recall that $T_a\leq T$, so if there exists $n$ in $N\backslash T$ such that all elements of the coset $nT$ have finite order, then also all elements of the non-trivial coset $nT_a$ will have finite order. We may thus shift our search to the spherical Weyl group $W=N/T$. In fact it does happen that $W$ can be realized as a subgroup of $W_a$, though we will not need to use this explicitly.

First some additional setup is necessary.

Let $N\geq N_0:=\gen m_{\al}(1)|\al\in\Phi\by$ and let $T_0:=\gen h_{\al}(-1)|\al\in\Phi\by$. Since $m_{\al}(-1)=m_{\al}(1)\I$, in fact $h_{\al}(-1)=m_{\al}(1)^{-2}$, so $T_0\leq N_0$. Also, by \cite{steinberg67:_chev_gps}*{Lemma~20(a)}, $m_{\al}(1)h_{\be}(-1)m_{\al}(1)\I=h_{s_{\al}(\be)}(-1)$, so $T_0\triangleleft N_0$.

\begin{lemma}
$W=N_0/T_0$.
\end{lemma}

\begin{proof} The proof is similar to \cite{steinberg67:_chev_gps}*{Lemma~22(b,c)}. Define a homomorphism $\phi_0:W\rightarrow N_0/T_0$ by $\phi_0(s_{\al})=T_0m_{\al}(1)$. To check this is well-defined one checks that the relations in $W$ are satisfied in $N_0/T_0$. Note that
$$T_0m_{\al}(1)=T_0h_{\al}(-1)m_{\al}(1)=T_0m_{\al}(-1)=T_0m_{\al}(1)\I$$
so for any $\al$, $s_{\al}^2\mapsto T_0m_{\al}(1)m_{\al}(1)\I=T_0$. Also,
$$s_{\al}s_{\be}s_{\al}\I s_{s_{\al}(\be)}\I\mapsto T_0m_{\al}(1)m_{\be}(1)m_{\al}(1)\I m_{s_{\al}(\be)}(1)\I.$$
But this is just $T_0$ since $m_{\al}(1)m_{\be}(1)m_{\al}(1)\I=m_{s_{\al}(\be)}(c)$ by \cite{steinberg67:_chev_gps}*{Lemma~20(b)}, where $c=\pm1$. If $c=-1$ one must also again use the fact that $T_0m_{\al}(-1)=T_0m_{\al}(1)$. These relations define $W$, so $\phi_0$ is well-defined, and is clearly surjective. Now suppose $w=s_{\al_1}\dots s_{\al_k}\mapsto T_0$, so $m_{\al_1}(1)\dots m_{\al_k}(1)\in T_0$. In particular, $m_{\al_1}(1)\dots m_{\al_k}(1)\in T$ and so by the proof of \cite{steinberg67:_chev_gps}*{Lemma~22(c)}, $w=1$. Thus, $\phi_0$ is an isomorphism.
\end{proof}

Now one can establish a criterion on $w$ whereby all representatives in $N$ will have finite order.

\begin{theorem}\label{chev_tor_criterion}
Let $W=N/T$ be the spherical Weyl group corresponding to the Chevalley group $\ger{g}(K)$. Let $w\in W$. Then the following are equivalent:
\begin{enumerate}[\rm (i)]
\item
As an orthogonal transformation of $E$, $w$ does not have eigenvalue 1.
\item
For any field $K$, every representative of $w$ in $N$ has finite order in $N$.
\end{enumerate}
\end{theorem}

\begin{proof}[Proof of the forward implication]
Let $w\in W$ have order $m$, and suppose 1 is not an eigenvalue of $w$. Since for any $v\in E$ we have
$$w(v+w(v)+\dots+w^{m-1}(v))=v+w(v)+\dots+w^{m-1}(v),$$
the hypothesis forces $1+w+\dots+w^{m-1}$ to be zero. Now, since $W=N_0/T_0$, there exists a representative $n_0\in N_0$ of $w$. Since $w^m=1$, $n_0^m\in T_0$. But $T_0$ is abelian and $h_{\al}(-1)^2=1$, so everything in $T_0$ has order 1 or 2. Since $m$ must divide the order of $n_0$, we know that $n_0$ has order $m$ or $2m$. Now let $h$ be any element of $T$. Say $n_0=m_{\al_1}(\ep_1)\dots m_{\al_k}(\ep_k)$ and $h=h_{\be_1}(\la_1)\dots h_{\be_{\ell}}(\la_{\ell})$. Here each $\ep_i$ is either 1 or -1, since $n_0\in N_0$ and $m_{\al}(1)\I=m_{\al}(-1)$. Note that since $n_0$ represents $w$ in $W$, we have that $w=s_{\al_1}\dots s_{\al_k}$.

Using the Chevalley relation
$$m_{\al}(1)h_{\be}(\la)m_{\al}(1)\I=m_{\al}(1)\I h_{\be}(\la)m_{\al}(1)=h_{s_{\al}(\be)}(\la),$$
one gets that $n_0h=h_{w(\be_1)}(\la_1)\dots h_{w(\be_{\ell})}(\la_{\ell})n_0$. Repeating this, one gets that

$$(n_0h)^m=\left(\prod_{i=1}^m\prod_{j=1}^{\ell}h_{w^i(\be_j)}(\la_j)\right)n_0^m=\left(\prod_{j=1}^{\ell}\prod_{i=1}^m h_{w^i(\be_j)}(\la_j)\right)n_0^m.$$
This last step follows since $T$ is abelian. Now, for any $j$ and for any weight $\gamma$,

$$\prod_{i=1}^m\la_j^{\langle\gamma,w^i(\be_j)\rangle}=\la_j^{\sum_{i=1}^m\langle\gamma,w^i(\be_j)\rangle}.$$
Since $w$ is an orthogonal transformation and $\{,\}$ is bilinear,

\begin{eqnarray*}\sum_{i=1}^m\langle\gamma,w^i(\be_j)\rangle &=& \sum_{i=1}^m2\{\gamma,w^i(\be_j)\}/\{w^i(\be_j),w^i(\be_j)\}\\&=&\sum_{i=1}^m2\{\gamma,w^i(\be_j)\}/\{\be_j,\be_j\}\\&=&\frac{2}{\{\be_j,\be_j\}}\left\{\gamma,\sum_{i=1}^mw^i(\be_j)\right\}=0
\end{eqnarray*}
for each $j$. Thus, $\displaystyle\prod_{i=1}^m\la_j^{\langle\gamma,w^i(\be_j)\rangle}=1$, regardless of the field $K$. By \cite{steinberg67:_chev_gps}*{Lemma~19(c)}, elements $h_{\al}(\la)$ of $T$ act on the weight space $V_{\gamma}$ via multiplication by $\la^{\langle\gamma,\al\rangle}$, and so in fact $\displaystyle\prod_{i=1}^m h_{w^i(\be_j)}(\la_j)=1$ for each $j$. One concludes that $(n_0h)^m=n_0^m$ for any $h\in T$. Since $w$ has order $m$, this implies that all representatives of $w$ must have the same order, and the result follows.
\end{proof}

Note that this proves something stronger. Every representative has finite order, and in fact they all have \emph{the same} order, either $m$ or $2m$. It is a quick exercise to check that such a $w$ exists, in fact any Coxeter element of $W$ will work, as seen in \cite{humphreys92:_refl_gps_and_cox_gps}*{Section~3.16~Lemma}. For completeness we will prove the reverse implication of Theorem~\ref{chev_tor_criterion}, though it is not needed to prove Proposition~\ref{mainprop}.

\begin{proof}[Proof of reverse implication]
Let $K=\Q$. Suppose $0\neq v\in E$ is a 1-eigenvector. Then $v+w(v)+\dots+w^{m-1}(v)=mv\neq0$, and so $1+w+\dots+w^{m-1}\neq0$ as a linear transformation. Since the roots span $E$, there exists a root $\be$ such that $\be+w(\be)+\dots+w^{m-1}(\be)\neq0$. Choose a representative $n_0\in N_0$ as before, so $n_0^m\in T_0$, say $n_0^m=h_{\al_1}(-1)\dots h_{\al_k}(-1)$. Then for any $r\in\mathbb{N}$,

$$(n_0h_{\be}(2))^{rm}=\left(\prod_{i=1}^{rm} h_{w^i(\be)}(2)\right)n_0^{rm}.$$
(We chose $K=\Q$ but in fact, any $K$ that is not an algebraic extension of a finite field will work; we just need an element with infinite multiplicative order; for $K=\Q$ we have used the number 2.) Suppose this equals 1 for some $r$. Then by \cite{steinberg67:_chev_gps}*{Lemma~19(c)}, for any weight $\gamma$ we have

$$1=\prod_{i=1}^{rm}2^{\langle\gamma,w^i(\be)\rangle}\prod_{j=1}^k(-1)^{r\langle\gamma,\al_k\rangle}=\pm\prod_{i=1}^{rm}2^{\langle\gamma,w^i(\be)\rangle}.$$
By the same argument as before, this equals

$$\pm2^{\frac{2}{\{\be,\be\}}\left\{\gamma,\sum_{i=1}^{rm}w^i(\be)\right\}}=\pm2^{\frac{2}{\{\be,\be\}}\left\{\gamma,r\sum_{i=1}^mw^i(\be)\right\}}.$$
The only way this can equal 1 is if $\displaystyle\left\{\gamma,r\sum_{i=1}^mw^i(\be)\right\}=0$. But since $\displaystyle r\sum_{i=1}^mw^i(\be)\neq0$, this is impossible, since one can always choose a weight $\gamma$ to be not orthogonal to $\displaystyle r\sum_{i=1}^mw^i(\be)$. Since the $rm$ are the only candidates for a finite order of $n_0h_{\be}(2)$, in fact it has infinite order. Since $n_0h_{\be}(2)$ is a representative of $w$, the theorem follows.
\end{proof}

Proposition~\ref{mainprop} now follows from Corollary~\ref{torcoset} and Theorem~\ref{chev_tor_criterion}, and the fact that Coxeter elements do not have eigenvalue 1.

\begin{remark}\label{rapinchuk_remark}
There is also a very nice, shorter proof due to A. Rapinchuk that all the representatives of $w$ have finite order if 1 is not an eigenvalue
of $w$ \cite{rapinchuk}. In fact his proof shows that each representative has order dividing
$m^2$, where $m=|w|$. Coupling the two proofs, we conclude that if $m$ is odd, then since representatives cannot have order $2m$ they must all have order $m$.
\end{remark}

\section{Torsionfree Weyl transitive subgroups of $\ger{g}(\Q_p)$}
\label{sec:weyl-trans-torfree}

One now has the tools to produce examples of Weyl transitive group actions on buildings that are not strongly transitive with respect to any apartment system, proving Proposition~\ref{prop2}. Let $(B_a,N)$ be the affine $BN$-pair of $G=\ger{g}(\Q_p)$ as described in Section~\ref{sec:chev_gps}. Think of $G$ as a subgroup of $\SL_d(\Q_p)$ for some $d$, and let $(\widetilde{B},\widetilde{N})$ be the usual affine $BN$-pair of $\SL_d(\Q_p)$ described in \cite{abramenko08:_build_theory_apps}*{Section~6.9}. The construction of $B_a$ shows that it is contained in $\widetilde{B}\cap G$. Also, $\widetilde{B}\cap G$ contains no nontrivial representatives of the affine Weyl group of $G$. Thus by looking at the affine Bruhat decomposition given by $(G,B_a)$ one sees that $B_a=G\cap\widetilde{B}$. This proves that $B_a$ is open in $G$, and so by Lemma~\ref{dense-subgroups} any dense subgroup of $G$ acts Weyl transitively on $\Delta$. Thus by Proposition~\ref{mainprop}, any dense, torsionfree subgroup of $G$ will act Weyl transitively but not strongly transitively on $\Delta$. We now exhibit a number of such subgroups, establishing Proposition~\ref{prop2}.

Let $\G=\ger{g}(\Z[\frac{1}{p}])$. While we technically have only been considering Chevalley groups over fields, this is allowed; see \cite{steinberg67:_chev_gps}*{Section~3}. Let $q$ be any nonzero integer prime to $p$, so it makes sense to reduce the entries of matrices in $\G$ mod $q$. Define the \emph{congruence subgroup} $\G_q$ to be $\G_q := \{A\in\G : A\equiv I_d\mod{q}\}$, where matrices are taken mod $q$ entry-wise. This is the kernel of the restriction to $\G$ of the natural group homomorphism $\SL_d(\Z[\frac{1}{p}])\rightarrow\SL_d(\Z[\frac{1}{p}]/q\Z[\frac{1}{p}])$, so it really is a subgroup. We will show that for any $q>2$ prime to $p$, $\G_q$ is both torsionfree and dense in $G$.

\begin{lemma}\label{densitylemma}

For any nonzero $q$ in $\Z$, $\G_q$ is dense in $G$.

\end{lemma}

\begin{proof}

Since the topological closure of $\Z[\frac{1}{p}]$ contains $\Z_p$ and $1/p$, $\Z[\frac{1}{p}]$ is dense in $\Q_p$. Also $\Q_p=q\Q_p$, so $q\Z[\frac{1}{p}]$ is dense in $\Q_p$. Thus for any $\al\in\Phi$, the set $\{x_{\al}(\la)\mid \la\in q\Z[\frac{1}{p}]\}$ is dense in $\{x_{\al}(\la)\mid \la\in \Q_p\}$. Since the latter set generates $\ger{g}(\Q_p)$, it now suffices to show that $x_{\al}(\la)\in\G_q$ for any $\al\in\Phi$, $\la\in q\Z[\frac{1}{p}]$. Since $x_{\al}(\la)$ has entries 1 on the main diagonal and entries congruent to $0\mod{q}$ off the main diagonal, it is clear that $x_{\al}(\la)\equiv I_d\mod{q}$. Also since $\G\cap\{A\in\SL_d(\Z[\frac{1}{p}])\mid A\equiv I_d\mod{q}\}=\G_q$ and $x_{\al}(\la)\in\G$, one sees that indeed $x_{\al}(\la)\in\G_q$.

\end{proof}


\begin{lemma}\label{torsionfree_lemma}

For any $q$ prime to $p$ with $q>2$, $\G_q$ is torsionfree.

\end{lemma}

\begin{proof}

Let $A\in\G_q$ with $A^r=I_d$. Suppose for a contradiction that $r>1$. By replacing $A$ with an appropriate power one may assume $r$ is prime. Let $B=A-I_d\in M_d(\Z[\frac{1}{p}])$, so $B\equiv 0\mod{q}$. Then $I_d=(I_d+B)^r$, and by the binomial expansion there exists $C\in M_d(\Z[\frac{1}{p}])$ such that $(I_d+B)^r=I_d+rB+CB^2$. Thus $rB=-CB^2$. Choose $s\geq1$ such that $B\equiv 0\mod{q^s}$ but $B\not\equiv 0\mod{q^{s+1}}$. Of course since $B\equiv 0\mod{q^s}$ one has $B^2\equiv 0\mod{q^{2s}}$, and so in fact $B^2\equiv 0\mod{q^{s+1}}$. One concludes that $q$ divides $r$. Since $r$ is prime and $q>2$, this implies that $q=r$ and $q$ is an odd prime.

We have
$$I_d=(I_d+B)^q=I_d+qB+\sum_{i=2}^q\binom{q}{i}B^i,$$
so
$$-qB=\sum_{i=2}^q\binom{q}{i}B^i.$$
Denote this last equality by $(\ast)$. Since $q$ is odd, $q$ divides $\displaystyle\binom{q}{2}$, and so the right-hand side of $(\ast)$ is congruent to zero mod $q^{s+2}$. Of course $s$ was chosen so that the left-hand side does \emph{not} satisfy that congruence, and so this is impossible.

Thus in fact $r=1$ and $\G_q$ is torsionfree.

\end{proof}

In this way, one sees that there are ``many" dense torsionfree subgroups of the Chevalley group $\ger{g}(\Q_p)$, proving Proposition~\ref{prop2}. 

\begin{remark}\label{other_base_fields}
A similar method can be used for other local fields $K$. The $K=\Q_p$ case is prototypical if $K$ has characteristic 0. If $K=\mathbb{F}_p((t))$, the above arguments can be modified to produce dense subgroups $H$ of $\ger{g}(K)$ that have only $p$-torsion. Assuming $p$ is chosen to not divide $2|W|$, this will yield the desired properties of the action of $H$ on the corresponding building.
\end{remark}

\renewcommand{\baselinestretch}{1}


\begin{bibdiv}
\begin{biblist}

\bib{abramenko07:_trans_properties}{article}{
  author={Abramenko, Peter},
  author={Brown, Kenneth S.},
  title={Transitivity properties for group actions on buildings},
  journal={Journal of Group Theory},
  volume={10},
  pages={267-277},
  date={2007},
}

\bib{abramenko08:_build_theory_apps}{book}{
  author={Abramenko, Peter},
  author={Brown, Kenneth S.},
  title={Buildings: Theory and Applications},
  series={Graduate Texts in Mathematics},
  volume={248},
  publisher={Springer-Verlag},
  address={New York},
  date={2008},
  isbn={978-0-387-78834-0},
}


\bib{humphreys92:_refl_gps_and_cox_gps}{book}{
  author={Humphreys, James E.},
  title={Reflection Groups and Coxeter Groups},
  series={Cambridge Studies in Advanced Mathematics},
  publisher={Cambridge University Press},
  address={Cambridge},
  date={1992},
  isbn={978-0-521-43613-7},
}

\bib{rapinchuk}{misc}{
	author={Rapinchuk, Andrei},
	title={Private correspondence},
}

\bib{steinberg67:_chev_gps}{book}{
  author={Steinberg, Robert},
  title={Lectures on Chevalley Groups},
  publisher={Yale University Press},
  date={1967},
}  

\bib{tits92:_twin_kac_moody}{article}{
  author={Tits, Jacques},
  title={Twin buildings and groups of Kac-Moody type},
  journal={London Mathematical Society Lecture Note Series},
  volume={165},
  date={1992},
  pages={249\ndash 286},
}

\bib{weiss09:_struct_affine_build}{book}{
	author={Weiss, Richard M.},
	title={The Structure of Affine Buildings},
	series={Annals of Mathematics Studies},
	publisher={Princeton University Press},
	address={Princeton},
	date={2009},
	isbn={978-0-691-13659-2},
}

\end{biblist}
\end{bibdiv}
\end{document}